%% file: main.tex
\documentclass[letterpaper,11pt]{article}
\usepackage[margin=1in]{geometry}
\usepackage{amsmath,amssymb,amsthm,mathtools,bm}
\usepackage{setspace,enumitem,hyperref,natbib}
\usepackage{authblk}

\newtheorem{theorem}{Theorem}
\newtheorem{proposition}{Proposition}
\newtheorem{lemma}{Lemma}

\newtheorem{assumption}{Assumption}

\newcommand{\R}{\mathbb{R}}
\newcommand{\N}{\mathbb{N}}

\newcommand{\E}{\mathbb{E}}
\newcommand{\Prob}{\mathbb{P}}
\newcommand{\boldp}{\mathbf{p}}
\newcommand{\boldR}{\mathbf{R}}
\newcommand{\boldone}{\mathbf{1}}
\newcommand{\boldphi}{\bm{\phi}}
\newcommand{\boldtheta}{\bm{\theta}}

\newcommand{\Var}{\operatorname*{Var}}
\newcommand{\Cov}{\operatorname*{Cov}}

\newcommand{\T}{\mathsf{T}}
\newtheorem{corollary}{Corollary}
\newcommand{\tbl}[2]{%
  \caption{#1}%
  \centering
  #2%
}
\newcommand{\RRAIPW}{\mathrm{RR\text{-}AIPW}}
\newcommand{\HRRAIPW}{\mathrm{H\text{-}RR\text{-}AIPW}}
\newcommand{\SepRRAIPW}{\mathrm{Sep\text{-}RR\text{-}AIPW}}

\title{Integration of external predictions for efficient estimation of the risk ratio}
\author[1,2]{Victoria Mezger}
\author[3]{Nils Kr{\"u}ger}
\author[1,3]{Georg Hahn}

\affil[1]{Harvard T.H. Chan School of Public Health, Boston, MA, USA}
\affil[2]{Heidelberg University, Department of Mathematics, Heidelberg, Germany}
\affil[3]{Brigham and Women's Hospital, Boston, MA, USA}

\date{
\vspace{-1em}
\textit{Contact:}
\texttt{victoria.mezger@math.uni-heidelberg.de} \quad
\texttt{nkruger1@bwh.harvard.edu} \quad
\texttt{ghahn@bwh.harvard.edu}
}

\begin{document}
\maketitle
\input{abstract}

\doublespacing
\input{introduction}
\input{methods_shortened}
\input{simulations2}
\input{discussion}

\appendix
\input{appendix_proofs_shortened}

\bibliographystyle{apalike}
\bibliography{main}
\end{document}

%% file: abstract.tex
\begin{abstract}
We consider the augmentation of randomized experiments or trials with data from observational studies for the purpose of improving statistical precision. In particular, we focus on the Hybrid Augmented Inverse Probability Weighting estimator, designed to integrate predictions from several foundation models while preserving valid statistical inference. In this article, we extend the Augmented Inverse Probability Weighting framework to the estimation of the risk ratio, the quotient of the absolute risk of an exposed to an unexposed group. Our approach allows one to use information from black-box foundation models trained on external and possibly unstructured data, yielding an estimator of the risk ratio whose asymptotic variance is never larger than the one of the default estimator based on experimental data alone.
\end{abstract}

%% file: introduction.tex
\section{Introduction}
\label{sec:introduction}
The augmentation of randomized experiments or trials with data from observational studies, for the purpose of improving statistical precision, is a timely and important research direction \citep{Lin2024}.

\citet{Bartolomeis2025} introduced the hybrid augmented inverse probability weighting (H-AIPW) estimator of the average treatment effect, which integrates predictions from several, possibly black-box, foundation models trained on external and unstructured data while preserving valid statistical inference: under minimal assumptions the resulting H-AIPW estimator is consistent and asymptotically normal, and its asymptotic variance never exceeds that of the standard estimator based on experimental data alone.

In this work, we extend this concept to AIPW estimation of the marginal risk ratio (RR-AIPW), also called the relative risk, defined as the ratio of the event probability in the exposed group to that in the unexposed group \citep{Sistrom2004, Carneiro2011}. We construct a direct hybrid risk ratio estimator and show that it retains the uniform domination of the variance already proven for H-AIPW. We further study an alternative construction that separately combines estimators for the treatment effect and control risk. In particular, we show that the second estimator is asymptotically normal, however it does not generally lead to a uniform domination of the variance because the sets of weights from the two estimators do not account for the joint covariance of the components.

Related work on integrating external data to improve estimation efficiency falls into three strands: frameworks that integrate foundation-model predictions into clinical-trial inference \citep{Bartolomeis2025, Poulet2025}; approaches that first test the compatibility of the experimental and external data, pooling them only if the test does not reject \citep{Luedtke2019, Hussain2023, Dahabreh2024, Bartolomeis2024a, Bartolomeis2024b}; and methods that debias a precise but biased observational estimate using an unbiased experimental estimator \citep{Kallus2018, Rosenman2023, Laan2024, Dang2025}.

%% file: methods_shortened.tex
\section{Risk ratio estimation}
\label{sec:riskratio}

\subsection{Setting and target parameter}
\label{sec:setup}
Suppose we observe independent copies $Z_i=(X_i,A_i,Y_i)$ $(i=1,\ldots,n)$ for $n \in \N$ individuals, where $X_i\in\R^d$ denotes baseline covariates in $d \in \N$ dimensions, $A_i\in\{0,1\}$ is treatment assignment and $Y_i\in\{0,1\}$ is the observed outcome. Let $Y_i(0)$ and $Y_i(1)$ denote binary potential outcomes, where $Y_i=Y_i(A_i)$. We assume that the following identification assumption holds.

\begin{assumption}[Identification]
\label{ass:simple-randomization}
The vector $\{X,Y(0),Y(1)\}$ is independent of $A$. The assignment probabilities satisfy $\pi_a=\Prob(A=a)>0$ for $a \in \{0,1\}$.
\end{assumption}

The treatment-specific risks are
\begin{equation}
\label{eq:pa}
p_a=\E\{Y(a)\}=\E(Y\mid A=a) \quad (a=0,1).
\end{equation}

We assume $p_0>0$ and target the marginal risk ratio, defined as
\begin{equation}
\label{eq:R}
R= \frac{p_1}{p_0}.
\end{equation}

Let $h^*(x,a)=\E(Y\mid X=x,A=a)$ denote the true outcome regression. We further assume we have an estimator $\hat h$ computed on external data satisfying the following assumption.

\begin{assumption}[Nuisance convergence]
\label{ass:nuisance}
The estimator $\hat{h}$ is trained on an auxiliary sample independent of $(Z_1,\ldots,Z_n)$, and there exists a square-integrable function $h^\dagger$ such that
\[
\|\hat h(\cdot,a)-h^\dagger(\cdot,a)\|_{L^2(P)}\longrightarrow_{P^*}0
\quad (a=0,1),
\]
where $P^*$ denotes the joint law of the trial and auxiliary samples.
\end{assumption}

The limit $h^\dagger$ need not equal the true $h^*$. In the simulations, the trial outcome regression is instead estimated by fixed-fold cross-fitting. The Supplementary Material proves the corresponding expansion under foldwise $L^2(P)$ convergence.

\subsection{RR-AIPW estimation of the treatment-specific risks}
\label{sec:AIPW_treatment}
For a square-integrable function $h:\R^d\times\{0,1\}\to\R$, define 
\begin{align}
\label{eq:psi1}
\psi^{(1)}(Z;h)
&=\frac{A}{\pi_1}\{Y-h(X,1)\}+h(X,1),\\
\label{eq:psi0}
\psi^{(0)}(Z;h)
&=\frac{1-A}{\pi_0}\{Y-h(X,0)\}+h(X,0).
\end{align}

\begin{proposition}[Unbiasedness]
\label{prop:unbiasedness}
For every square-integrable $h$, $\E\{\psi^{(a)}(Z;h)\}=p_a$ for $a=0,1$.
\end{proposition}

As shown in Proposition~\ref{prop:unbiasedness}, misspecification of $h$ affects the variance but not the expectation of the score. The corresponding AIPW risk estimators are
\begin{equation}
\label{eq:phat}
\hat p_a(h)=n^{-1}\sum_{i=1}^n\psi^{(a)}(Z_i;h)
\quad (a=0,1),
\end{equation}
for which the following joint asymptotic result holds.

\begin{lemma}[Joint asymptotic linearity]
\label{lem:aipw-risk}
Under Assumptions~\ref{ass:simple-randomization} and \ref{ass:nuisance}, with $\hat{\boldp}(\hat h)=\{\hat p_1(\hat h),\hat p_0(\hat h)\}^{\T}$ and $\boldp=(p_1,p_0)^{\T}$,
\[
\sqrt n\{\hat{\boldp}(\hat h)-\boldp\}
=
\frac{1}{\sqrt n}\sum_{i=1}^n
\begin{pmatrix}
\psi^{(1)}(Z_i;h^\dagger)-p_1\\
\psi^{(0)}(Z_i;h^\dagger)-p_0
\end{pmatrix}
+o_{P^*}(1)
\rightsquigarrow N(0,\Gamma),
\]
where
\(
\Gamma=\Cov\left\{
\begin{pmatrix}
\psi^{(1)}(Z;h^\dagger)\\
\psi^{(0)}(Z;h^\dagger)
\end{pmatrix}
\right\}.
\)
\end{lemma}

Replacing $h^\dagger$ by $\hat h$ is therefore first-order negligible. The limiting influence function depends on the outcome-regression estimator only through its probability limit. When $h^\dagger=h^*$, the resulting RR-AIPW influence function is efficient.

\subsection{The RR-AIPW estimator}
\label{sec:direct}
Define the RR-AIPW estimator as the ratio of the AIPW risk estimators of eq.~\eqref{eq:phat}, that is
\begin{equation}
\label{eq:RRAIPW}
\hat R_{\RRAIPW}(h)=\frac{\hat p_1(h)}{\hat p_0(h)}.
\end{equation}
The RR-AIPW estimator is consistent and asymptotically normal as shown in the next result. 

\begin{proposition}[Asymptotic behaviour of the RR-AIPW estimator]
\label{prop:rr-aipw}
Under Assumptions~\ref{ass:simple-randomization} and \ref{ass:nuisance},
\[
\sqrt n\{\hat R_{\RRAIPW}(\hat h)-R\}
\rightsquigarrow N(0,V_R^{h^\dagger}),
\]
where $V_R^{h^\dagger}=\Var\{\phi_R(Z;h^\dagger)\}$ and
\begin{align}
\phi_R(Z;h) = \frac{\psi^{(1)}(Z;h)-p_1}{p_0} -
\frac{p_1}{p_0^2}\{\psi^{(0)}(Z;h)-p_0\}.
\label{eq:phiR}
\end{align}
\end{proposition}

For a square-integrable prediction function $h$, define
\begin{equation}
\label{eq:prediction-error}
\epsilon[h](x)=h(x,1)-h^*(x,1)
+\frac{\pi_1R}{\pi_0}\{h(x,0)-h^*(x,0)\}.
\end{equation}

\begin{proposition}[Variance decomposition]
\label{prop:vardecomp}
Under Assumption~\ref{ass:simple-randomization}, for every square-integrable $h$,
\[
\phi_R(Z;h)=\phi_R(Z;h^*)+
\frac{1-A/\pi_1}{p_0}\epsilon[h](X).
\]
The two terms on the right are uncorrelated. Consequently,
\begin{equation}
\label{eq:vardecomp}
V_R^h=V_R^{h^*}+
\frac{\pi_0}{\pi_1p_0^2}\|\epsilon[h]\|_{L^2(P)}^2.
\end{equation}
\end{proposition}

In eq.~\eqref{eq:vardecomp}, $V_R^{h^*}$ is the semiparametric efficiency bound for $R$ in the randomized-trial model with known assignment probabilities \citep{tsiatis2006semiparametric}. Thus, $V_R^h=V_R^{h^*}$ if and only if $\epsilon[h]=0$ almost surely. Under equal allocation, the relevant error is $h(\cdot,1)-h^*(\cdot,1)+R\{h(\cdot,0)-h^*(\cdot,0)\}$. Correct prediction in both arms is sufficient, but not necessary, for efficiency.

\subsection{Direct hybrid RR-AIPW estimation}
\label{sec:hybrid_riskratio}

Suppose that, in addition to $\hat h$, there are $k \in \N$ square-integrable external outcome prediction functions $f_1,\ldots,f_k$. Each such function provides predictions under both treatment levels, is trained without the target-trial outcomes and is fixed with respect to $n$. Define the candidate RR-AIPW estimators
\begin{align*}
\hat R_{\RRAIPW,1} &= \hat R_{\RRAIPW}(\hat h),\\
\hat R_{\RRAIPW,j+1} &= \hat R_{\RRAIPW}(f_j)
\quad (j=1,\ldots,k),
\end{align*}
and let $\hat{\boldR}=(\hat R_{\RRAIPW,1},\ldots,\hat R_{\RRAIPW,k+1})^{\T}$. For the limiting influence functions, set $g_1=h^\dagger$ and $g_{j+1}=f_j$. For $\ell=1,\ldots,k+1$, let $\phi_{R,\ell}(Z)=\phi_R(Z;g_\ell)$, and define the stacked function
\[
\boldphi_R(Z)=\{\phi_{R,1}(Z),\ldots,\phi_{R,k+1}(Z)\}^{\T},
\qquad
\Omega=\Cov\{\boldphi_R(Z)\}.
\]
For any $\lambda \in \Lambda$ in the standard probability simplex in $k$ dimensions, given by $\Lambda=\{\lambda\in\R^{k+1}:\boldone^{\T}\lambda=1\}$, the hybrid RR-AIPW estimator $\hat R_\lambda=\lambda^{\T}\hat{\boldR}$ has asymptotic variance $\lambda^{\T}\Omega\lambda$. If $\Omega$ is nonsingular, the oracle weights minimizing the variance are
\begin{equation}
\label{eq:oracle-weight}
\lambda^*
=
\frac{\Omega^{-1}\boldone}{\boldone^{\T}\Omega^{-1}\boldone}.
\end{equation}

In order to implement this framework, let $\tilde p_a$ be a common consistent estimator of $p_a$ $(a=0,1)$, taken to be the cross-fitted trial-only AIPW risk estimator. Let $\hat g_{1i}$ denote the cross-fitted trial prediction used for observation $i$, and let $\hat g_{\ell i}=f_{\ell-1}$ for $\ell\geq2$. Define
\[
\hat\phi_{R,\ell i}
=
\frac{\psi^{(1)}(Z_i;\hat g_{\ell i})-\tilde p_1}{\tilde p_0}
-
\frac{\tilde p_1}{\tilde p_0^2}
\{\psi^{(0)}(Z_i;\hat g_{\ell i})-\tilde p_0\}.
\]
Writing $m=k+1$, $\hat{\boldphi}_{Ri}=(\hat\phi_{R,1i},\ldots,\hat\phi_{R,mi})^{\T}$ and $\bar{\boldphi}_R=n^{-1}\sum_i\hat{\boldphi}_{Ri}$, define the covariance estimator
\begin{equation}
\label{eq:Omegahat}
\hat\Omega
=
\frac{1}{n-1}\sum_{i=1}^n
(\hat{\boldphi}_{Ri}-\bar{\boldphi}_R)
(\hat{\boldphi}_{Ri}-\bar{\boldphi}_R)^{\T}.
\end{equation}
For numerical stability, we will consider the vanishing-ridge covariance estimator
\begin{equation}
\label{eq:Omegareg}
\hat\Omega_{\rm reg}
=
\hat\Omega+n^{-1/2}\bar\omega_R I_m,
\qquad
\bar\omega_R=m^{-1}\operatorname{tr}(\hat\Omega).
\end{equation}
If $\hat\Omega\to_P\Omega$, then $\bar\omega_R=O_P(1)$ and $\hat\Omega_{\rm reg}\to_P\Omega$.

\begin{theorem}[Direct hybrid RR-AIPW estimator]
\label{thm:hybrid-rr}
Suppose the preceding conditions of Section~\ref{sec:hybrid_riskratio} hold, let $\Omega$ be nonsingular and $\hat\Omega\to_P\Omega$. Let
\[
\hat\lambda
=
\frac{\hat\Omega_{\rm reg}^{-1}\boldone}
{\boldone^{\T}\hat\Omega_{\rm reg}^{-1}\boldone},
\qquad
\hat R_{\HRRAIPW}=\hat\lambda^{\T}\hat{\boldR}.
\]
Then
\[
\sqrt n(\hat R_{\HRRAIPW}-R)
\rightsquigarrow
N(0,(\lambda^*)^{\T}\Omega\lambda^*),
\]
and
\[
(\lambda^*)^{\T}\Omega\lambda^*
\leq \Omega_{\ell\ell} \qquad (\ell=1,\ldots,k+1).
\]
\end{theorem}

By Theorem~\ref{thm:hybrid-rr}, the direct hybrid RR-AIPW estimator has asymptotic variance no greater than that of the trial-only RR-AIPW estimator. Arbitrarily inaccurate external predictions may receive negligible or negative weight but do not invalidate inference or increase the oracle asymptotic variance.

\subsection{Componentwise hybrid RR-AIPW estimation}
\label{sec:alternative}
The risk ratio also satisfies $R=1+\theta/p_0$, where $\theta=p_1-p_0$ is the marginal risk difference. For any $h$, define
\[
\psi^\theta(Z;h)=\psi^{(1)}(Z;h)-\psi^{(0)}(Z;h),
\qquad
\hat\theta_{\rm AIPW}(h)=n^{-1}\sum_{i=1}^n\psi^\theta(Z_i;h),
\]
and the \textit{plug-in} estimator of the risk ratio as
\[
\hat R_{\rm plug\text{-}RR\text{-}AIPW}(h)
=1+\frac{\hat\theta_{\rm AIPW}(h)}{\hat p_0(h)}.
\]
As shown next, the plug-in estimator is equivalent to the previous RR-AIPW estimator of eq.~\eqref{eq:RRAIPW}.

\begin{proposition}[Equivalence]
\label{prop:direct-plugin-equivalence}
For every $h$ such that $\hat p_0(h)\neq0$,
$\hat R_{\rm plug\text{-}RR\text{-}AIPW}(h)=\hat R_{\RRAIPW}(h)$.
\end{proposition}

Although the two representations are identical for a common prediction function, they motivate a componentwise hybrid RR-AIPW estimator when distinct weights are used for $\theta$ and $p_0$. Define
\begin{align*}
\hat{\boldtheta}
&=\{\hat\theta_{\rm AIPW}(\hat h),\hat\theta_{\rm AIPW}(f_1),\ldots,\hat\theta_{\rm AIPW}(f_k)\}^{\T},\\
\hat{\boldp}_0
&=\{\hat p_0(\hat h),\hat p_0(f_1),\ldots,\hat p_0(f_k)\}^{\T}.
\end{align*}
With $g_\ell$ as in Section~\ref{sec:hybrid_riskratio}, let
\begin{align*}
\boldphi_\theta(Z)
&=
\{\psi^\theta(Z;g_1)-\theta,\ldots,
\psi^\theta(Z;g_{k+1})-\theta\}^{\T},\\
\boldphi_0(Z)
&=
\{\psi^{(0)}(Z;g_1)-p_0,\ldots,
\psi^{(0)}(Z;g_{k+1})-p_0\}^{\T},
\end{align*}
and define
\[
\Sigma_\theta=\Cov\{\boldphi_\theta(Z)\},
\quad
\Sigma_0=\Cov\{\boldphi_0(Z)\},
\quad
\Sigma_{\theta0}=\Cov\{\boldphi_\theta(Z),\boldphi_0(Z)\}.
\]
If $\Sigma_\theta$ and $\Sigma_0$ are nonsingular, their respective oracle weights are given by
\[
\lambda_\theta^*
=
\frac{\Sigma_\theta^{-1}\boldone}{\boldone^{\T}\Sigma_\theta^{-1}\boldone},
\qquad
\lambda_0^*
=
\frac{\Sigma_0^{-1}\boldone}{\boldone^{\T}\Sigma_0^{-1}\boldone}.
\]
Let $\hat\Sigma_\theta$ and $\hat\Sigma_0$ denote the corresponding empirical covariance matrices, formed analogously to eq.~\eqref{eq:Omegahat} and regularized as in eq.~\eqref{eq:Omegareg} using scaled mean diagonal elements. Let $\hat\lambda_\theta,\hat\lambda_0\in\Lambda$ be the weights computed from these regularized covariance matrices, then the estimator
\[
\hat R_{\SepRRAIPW}
=
1+
\frac{\hat\lambda_\theta^{\T}\hat{\boldtheta}}
{\hat\lambda_0^{\T}\hat{\boldp}_0},
\]
again satisfies an asymptotic normality result.

\begin{theorem}[Componentwise hybrid RR-AIPW estimator]
\label{thm:separate-plug-in}
Suppose that the preceding conditions of Section~\ref{sec:alternative} hold. Let $\Sigma_\theta$ and $\Sigma_0$ be nonsingular, and assume $\hat\Sigma_\theta\longrightarrow_P\Sigma_\theta$ as well as $\hat\Sigma_0\longrightarrow_P\Sigma_0$.
Then,
\[
\sqrt n(\hat R_{\SepRRAIPW}-R)
\rightsquigarrow N(0,\widetilde V_{\rm sep}),
\]
where
\[
\widetilde V_{\rm sep}
=
\Var\left\{
\frac{(\lambda_\theta^*)^{\T}\boldphi_\theta(Z)}{p_0}
-
\frac{\theta(\lambda_0^*)^{\T}\boldphi_0(Z)}{p_0^2}
\right\}
=
\frac{(\lambda_\theta^*)^{\T}\Sigma_\theta\lambda_\theta^*}{p_0^2} +
\frac{\theta^2(\lambda_0^*)^{\T}\Sigma_0\lambda_0^*}{p_0^4} -
\frac{2\theta(\lambda_\theta^*)^{\T}\Sigma_{\theta0}\lambda_0^*}{p_0^3}.
\]
\end{theorem}

The componentwise hybrid RR-AIPW estimator does not generally satisfy a variance-domination guarantee for the risk ratio. Its two weight vectors minimize the variances of the risk difference and control risk separately, without optimizing the cross-covariance term in $\widetilde V_{\rm sep}$.

%% file: simulations2.tex
\section{Simulation study}
\label{sec:simulation}

\subsection{Design}
We conducted 2000 Monte Carlo replications of randomized trials with $n=500$ individuals, generating $d=10$ dimensional covariates according to
\[
X_i\sim N(0,I_d),\qquad
A_i\sim\operatorname{Bernoulli}(1/2),\qquad
Y_i\mid X_i,A_i=a
\sim
\operatorname{Bernoulli}\{m_a(X_i)\},
\]
where $I_d$ is the $(d \times d)$-identity matrix, $m_a(x)=\operatorname{expit}\{\eta_a(x)\}$, and $\eta_a$ is given below. Let
\[
\beta=(0.35,-0.25,0.20,0.15,-0.10,0.10,0,\ldots,0)^{\T} \in \R^d.
\]
We considered two scenarios. A linear scenario, given by $\eta_0(x)=\alpha_0+\beta^{\T}x$ and $\eta_1(x)=\eta_0(x)+0.35$; and a nonlinear scenario, given by
\begin{align*}
\eta_0(x)&=\alpha_0+\beta^{\T}x +0.45\sin(x_1)+0.30(x_2^2-1)-0.20x_3x_4,\\
\eta_1(x)&=\eta_0(x)+0.35+0.25\,\mathbb{I}(x_1>0)-0.15x_2.
\end{align*}

The arm-specific trial logistic regressions were correctly specified in the linear setting and misspecified in the nonlinear setting. The intercept was calibrated to yield $p_0 \simeq 0.25$; the results of Monte Carlo integration with 200000 observations were taken as true values of the risk ratio, yielding 1.2692 and 1.3608 in the linear and non-linear settings, respectively.

Each replication included three candidate outcome prediction functions. The trial candidate used five-fold cross-fitted, arm-specific logistic regression with all ten covariates as linear terms.

For the stronger external candidate, define $u(x)=\sin(x_5)+0.5x_6$ and $\sigma_u^2=(1-e^{-2})/2+1/4$, and set
\(
f_{{\rm strong},a}(x)
=
\operatorname{expit}\{\eta_a(x)+0.15u(x)/\sigma_u\}
\)
for $a=0,1$ and $x \in \R^d$.

For the moderate external candidate, let
\[
\widetilde\beta=(0.25,-0.18,0.12,0.08,0,\ldots,0)^{\T} \in \R^d,
\]
and define $f_{{\rm moderate},a}(x)=\operatorname{expit}\{\eta_{a,{\rm mod}}(x)\}$, where $\eta_{0,{\rm mod}}(x)=\alpha_0+\widetilde\beta^{\T}x$ and $\eta_{1,{\rm mod}}(x)=\eta_{0,{\rm mod}}(x)+0.30+0.10x_1$.

Denote $(\hat g_1,\hat g_2,\hat g_3)
=(\hat h,f_{\rm strong},f_{\rm moderate})$. For each candidate $\hat g_j$, treatment-specific risks were estimated using the AIPW scores with $\pi_0=\pi_1=1/2$, yielding
\(
\hat R_{\RRAIPW}(\hat g_j)=\hat p_1(\hat g_j) / \hat p_0(\hat g_j).
\)
We compared the trial-only estimator $\hat R_{\RRAIPW}(\hat h)$, the direct hybrid estimator
\[
\hat R_{\HRRAIPW}
=
\sum_{j=1}^3\hat\lambda_j^R\hat R_{\RRAIPW}(\hat g_j),
\]
and the componentwise hybrid estimator
\[
\hat R_{\SepRRAIPW}
=
1+
\frac{\sum_{j=1}^3\hat\lambda_j^\theta
\{\hat p_1(\hat g_j)-\hat p_0(\hat g_j)\}}
{\sum_{j=1}^3\hat\lambda_j^0\hat p_0(\hat g_j)}.
\]
For the direct estimator, we consider
\[
\hat\lambda^R=
\frac{\hat\Omega_{\rm reg}^{-1}\boldone}
{\boldone^{\T}\hat\Omega_{\rm reg}^{-1}\boldone},
\qquad
\hat\Omega_{\rm reg}
=
\hat\Omega+n^{-1/2}\bar\omega_R I_3,
\qquad
\bar\omega_R=3^{-1}\operatorname{tr}(\hat\Omega).
\]
The componentwise covariance matrices were regularized analogously. At $n=500$, the ridge fraction was $n^{-1/2}=0.0447$.

All estimated risk ratio influence functions used common plug-in values of $p_1$ and $p_0$ given by the trial-only cross-fitted AIPW estimates. Wald standard errors were obtained from the sample standard deviations of the combined estimated influence functions divided by $n^{1/2}$. No weight fallback or denominator truncation occurred in any replication.

We calculated the mean estimate, Monte Carlo standard deviation, empirical coverage of nominal 95\% Wald intervals, and implied sample size reduction. For a hybrid estimator $\hat R_{\rm hybrid}$, the percent sample size reduction (SSR) is estimated as
\[
\widehat{\rm SSR}(\%)
=
100\left\{
1-
\frac{\widehat V_{\rm MC}(\hat R_{\rm hybrid})}
{\widehat V_{\rm MC}(\hat R_{\RRAIPW})}
\right\},
\qquad
\hat n_{\rm hybrid}
=
n\,
\frac{\widehat V_{\rm MC}(\hat R_{\rm hybrid})}
{\widehat V_{\rm MC}(\hat R_{\RRAIPW})},
\]
where $\widehat V_{\rm MC}(\hat R_{\RRAIPW})$ denotes the Monte Carlo variance. The sample size interpretation assumes inverse-sample-size variance scaling. Paired percentile intervals for SSR were obtained from 1000 bootstrap resamples of the 2000 Monte Carlo replications.

\subsection{Results}
\begin{table}[h]
\tbl{Primary RR-AIPW simulation results for $n=500$. Sample size reduction (SSR) is relative to the trial-only RR-AIPW estimator}
{{\scriptsize\setlength{\tabcolsep}{3pt}
\begin{tabular}{llrrrrlr}
\hline
Setting & Estimator & Truth & Mean & SD & Coverage &
SSR, \% (95\% CI) & Required $n$\\
\hline
Linear & Trial-only RR-AIPW
    & 1.269 & 1.278 & 0.1815 & 0.956
    & 0.00 & 500\\
& Direct hybrid RR-AIPW
    & 1.269 & 1.278 & 0.1791 & 0.958
    & 2.65 (1.15, 4.08) & 487\\
& Componentwise hybrid RR-AIPW
    & 1.269 & 1.278 & 0.1791 & 0.957
    & 2.68 (1.23, 4.04) & 487\\[3pt]
Nonlinear & Trial-only RR-AIPW
    & 1.361 & 1.377 & 0.1888 & 0.953
    & 0.00 & 500\\
& Direct hybrid RR-AIPW
    & 1.361 & 1.377 & 0.1831 & 0.948
    & 5.94 (4.26, 7.53) & 470\\
& Componentwise hybrid RR-AIPW
    & 1.361 & 1.376 & 0.1829 & 0.951
    & 6.12 (4.40, 7.72) & 469\\
\hline
\end{tabular}}}
\label{tab:rr-primary-n500}
\end{table}

Table~\ref{tab:rr-primary-n500} reports our results. In the correctly specified linear setting, the direct hybrid RR-AIPW estimator achieved an implied SSR of 2.65\% (95\% bootstrap interval, $[1.15\%,4.08\%]$). The componentwise hybrid RR-AIPW estimator achieved an SSR of 2.68\% ($[1.23\%,4.04\%]$). Under inverse-sample-size variance scaling, these results correspond to requiring approximately 487 participants to match the precision of a trial-only RR-AIPW analysis with 500 participants.

In the nonlinear setting, the implied SSRs were 5.94\% ($[4.26\%,7.53\%]$) for the direct hybrid RR-AIPW estimator and 6.12\% ($[4.40\%,7.72\%]$) for the componentwise estimator, corresponding to approximately 470 and 469 participants, respectively. Empirical coverage remained close to the nominal 95\% level for all estimators. The small positive bias was consistent with the usual $O(n^{-1})$ second-order bias of ratio estimators. The mean regularized covariance condition numbers for the direct hybrid weights were approximately 41 and 29, and the maximum absolute weights across replications were 0.590 and 0.896.

%% file: discussion.tex
\section{Discussion}
\label{sec:discussion}
The direct hybrid estimator extends the variance-domination guarantee of \citet{Bartolomeis2025} from the average treatment effect to the risk ratio, allowing one to incorporate external predictions of arbitrary quality without degrading asymptotic precision. Our future work will assess whether the precision gains obtained by the proposed RR-AIPW estimators translate to clinical applications. This requires evaluation in completed randomized trials using prediction models trained on external clinical data, such as experimental data from historical trials or observational data from registries, electronic health records or claims. Such studies will be important for determining how performance is affected by imperfect transportability, differences in outcome definitions and covariate distributions, and lower event rates than those considered in the simulations of this work.

%% file: appendix_proofs_shortened.tex
\section{Proofs}
\label{sec:supp-proofs}

\subsection{Unbiasedness of the treatment-specific risk scores}

\begin{proof}[Proof of Proposition~\ref{prop:unbiasedness}]
We prove the result for $a=1$; the argument for $a=0$ is analogous. Conditional on $X$,
\begin{align*}
\E\{\psi^{(1)}(Z;h)\mid X\}
&=\pi_1^{-1}\E[A\{Y-h(X,1)\}\mid X]+h(X,1)\\
&=\E\{Y-h(X,1)\mid X,A=1\}+h(X,1)\\
&=\E(Y\mid X,A=1)=\E\{Y(1)\mid X\}.
\end{align*}
The second equality uses $\Prob(A=1\mid X)=\pi_1$, and the final equality follows from randomization and consistency. Taking expectations gives $\E\{\psi^{(1)}(Z;h)\}=p_1$.
\end{proof}

\subsection{Joint asymptotic linearity with an independent training sample}

\begin{proof}[Proof of Lemma~\ref{lem:aipw-risk}]
For fixed $h^\dagger$, Proposition~1 gives
\[
\sqrt n\{\hat{\boldp}(h^\dagger)-\boldp\}
=
\frac{1}{\sqrt n}\sum_{i=1}^n
\begin{pmatrix}
\psi^{(1)}(Z_i;h^\dagger)-p_1\\
\psi^{(0)}(Z_i;h^\dagger)-p_0
\end{pmatrix}.
\]
It remains to show that replacing $h^\dagger$ by $\hat h$ is first-order negligible. We give the argument for $a=1$. Put $\Delta_1(x)=\hat h(x,1)-h^\dagger(x,1)$. Then
\[
\sqrt n\{\hat p_1(\hat h)-\hat p_1(h^\dagger)\}
=
\frac{1}{\sqrt n}\sum_{i=1}^n
\left(1-\frac{A_i}{\pi_1}\right)\Delta_1(X_i).
\]
Conditional on $\hat h$, the summands are independent and have conditional mean zero. Their conditional variance is
\[
\E\left[
\left(1-\frac{A}{\pi_1}\right)^2\Delta_1(X)^2
\mathrel{\Big|}\hat h
\right]
=
\frac{\pi_0}{\pi_1}
\|\hat h(\cdot,1)-h^\dagger(\cdot,1)\|_{L^2(P)}^2.
\]
For $\varepsilon>0$, conditional Chebyshev gives a probability bounded by
\[
B_n=\varepsilon^{-2}\frac{\pi_0}{\pi_1}
\|\hat h(\cdot,1)-h^\dagger(\cdot,1)\|_{L^2(P)}^2.
\]
The unconditional probability is bounded by $\E^*\{\min(1,B_n)\}$. Since $B_n\to_{P^*}0$ and $0\leq\min(1,B_n)\leq1$, this expectation converges to zero. Hence
\[
\sqrt n\{\hat p_1(\hat h)-\hat p_1(h^\dagger)\}=o_{P^*}(1).
\]
The same argument for $a=0$ gives
\[
\sqrt n\{\hat p_0(\hat h)-\hat p_0(h^\dagger)\}=o_{P^*}(1).
\]
Stacking the coordinates yields the asserted asymptotic linear expansion. The multivariate central limit theorem gives the stated normal limit because the vector summands are independent, mean zero and square-integrable.
\end{proof}

\subsection{Cross-fitted nuisance estimation}

Let $I_1,\ldots,I_V$ be a random partition independent of the observations, with fixed $V$ and $|I_v|/n\to\nu_v\in(0,1)$. Let $\hat h^{(-v)}$ be trained without observations in $I_v$, and use $\hat h_i=\hat h^{(-v)}$ for $i\in I_v$. Suppose that, for every $v$ and $a$,
\[
\|\hat h^{(-v)}(\cdot,a)-h^\dagger(\cdot,a)\|_{L^2(P)}\longrightarrow_P0.
\]
Define
\[
\hat p_a^{\rm cf}
=n^{-1}\sum_{v=1}^V\sum_{i\in I_v}
\psi^{(a)}\{Z_i;\hat h^{(-v)}\}.
\]

\begin{lemma}[Cross-fitted expansion]
Under these conditions,
\[
\sqrt n
\begin{pmatrix}
\hat p_1^{\rm cf}-p_1\\
\hat p_0^{\rm cf}-p_0
\end{pmatrix}
=
\frac{1}{\sqrt n}\sum_{i=1}^n
\begin{pmatrix}
\psi^{(1)}(Z_i;h^\dagger)-p_1\\
\psi^{(0)}(Z_i;h^\dagger)-p_0
\end{pmatrix}
+o_P(1).
\]
\end{lemma}

\begin{proof}
For $a=1$, the remainder is
\[
\sum_{v=1}^V
\frac{1}{\sqrt n}\sum_{i\in I_v}
\left(1-\frac{A_i}{\pi_1}\right)
\{\hat h^{(-v)}(X_i,1)-h^\dagger(X_i,1)\}.
\]
For a fixed fold $v$, multiply and divide by $|I_v|^{1/2}$. Conditional on the training data outside $I_v$, the normalized sum over $I_v$ has mean zero and variance at most
\[
\frac{\pi_0}{\pi_1}
\|\hat h^{(-v)}(\cdot,1)-h^\dagger(\cdot,1)\|_{L^2(P)}^2=o_P(1).
\]
The additional factor $(|I_v|/n)^{1/2}$ is bounded. Thus each fold contribution is $o_P(1)$, and their sum is $o_P(1)$ because $V$ is fixed. The argument for $a=0$ is identical. Adding the fixed-$h^\dagger$ empirical-process term gives the expansion, and the multivariate central limit theorem gives the limit distribution.
\end{proof}

\subsection{Consistency of the plug-in covariance estimators}
Let $\tilde p_a\to_P p_a$ for $a=0,1$, let $\hat g_{1i}$ be the cross-fitted trial prediction used for observation $i$, and let the remaining candidate functions be fixed. Under the foldwise $L^2(P)$ convergence conditions above,
\[
\frac1n\sum_{i=1}^n
\|\hat{\boldphi}_{Ri}-\boldphi_R(Z_i)\|^2=o_P(1).
\]
Indeed, the AIPW scores are linear in the candidate prediction functions, while the coefficients in the risk-ratio influence function are continuous in $(p_1,p_0)$ on a neighbourhood with $p_0>0$. The displayed empirical $L^2$ convergence therefore follows from the foldwise argument and consistency of the common plug-in risks. Since the candidate dimension is fixed and $\boldphi_R(Z)$ is square-integrable, Cauchy--Schwarz and the law of large numbers give
\[
\hat\Omega-\Omega=o_P(1).
\]
The identical argument applied to $\boldphi_\theta$ and $\boldphi_0$ yields $\hat\Sigma_\theta\to_P\Sigma_\theta$ and $\hat\Sigma_0\to_P\Sigma_0$.

\subsection{The RR-AIPW estimator}

\begin{proof}[Proof of Proposition~\ref{prop:rr-aipw}]
The joint expansion gives
\[
\sqrt n
\begin{pmatrix}
\hat p_1(\hat h)-p_1\\
\hat p_0(\hat h)-p_0
\end{pmatrix}
=
\frac{1}{\sqrt n}\sum_{i=1}^n
\begin{pmatrix}
\psi^{(1)}(Z_i;h^\dagger)-p_1\\
\psi^{(0)}(Z_i;h^\dagger)-p_0
\end{pmatrix}
+o_{P^*}(1).
\]
The map $g(u,v)=u/v$ is continuously differentiable near $(p_1,p_0)$ because $p_0>0$, with gradient $\nabla g(p_1,p_0)=(p_0^{-1},-p_1p_0^{-2})$. Using the definition of $\phi_R$, the delta method therefore yields
\[
\sqrt n\{\hat R_{\RRAIPW}(\hat h)-R\}
=\frac{1}{\sqrt n}\sum_{i=1}^n\phi_R(Z_i;h^\dagger)+o_{P^*}(1).
\]
The influence function has mean zero and finite variance, so the conclusion follows from the central limit theorem and Slutsky's theorem.

Writing $\phi_a(Z;h)=\psi^{(a)}(Z;h)-p_a$ $(a=0,1)$, the asymptotic variance is
\[
V_R^h
=
\frac{\Var\{\phi_1(Z;h)\}}{p_0^2}
+
\frac{p_1^2\Var\{\phi_0(Z;h)\}}{p_0^4}
-
\frac{2p_1\Cov\{\phi_1(Z;h),\phi_0(Z;h)\}}{p_0^3}.
\]
\end{proof}
\subsection{Variance decomposition}
\begin{proof}[Proof of Proposition~\ref{prop:vardecomp}]
For $a=0,1$, put $\delta_a(x)=h(x,a)-h^*(x,a)$. Subtraction of the risk scores gives
\begin{align*}
\psi^{(1)}(Z;h)-\psi^{(1)}(Z;h^*)&=(1-A/\pi_1)\delta_1(X),\\
\psi^{(0)}(Z;h)-\psi^{(0)}(Z;h^*)&=-\frac{\pi_1}{\pi_0}(1-A/\pi_1)\delta_0(X).
\end{align*}
Substitution in the definition of $\phi_R$ gives
\[
\phi_R(Z;h)-\phi_R(Z;h^*)
=\frac{1-A/\pi_1}{p_0}
\left\{\delta_1(X)+\frac{\pi_1R}{\pi_0}\delta_0(X)\right\}
=:D(Z).
\]
By randomization, $\E(D\mid X)=0$. To prove orthogonality, write
\begin{align*}
\phi_R(Z;h^*)
={}&\frac{A}{\pi_1p_0}\{Y-h^*(X,1)\}
-\frac{R(1-A)}{\pi_0p_0}\{Y-h^*(X,0)\}\\
&+\frac{h^*(X,1)-Rh^*(X,0)}{p_0}.
\end{align*}
The last term is a function of $X$, so its product with $D$ has expectation zero. The first two terms have conditional expectation zero given $(X,A)$, whereas $D$ is measurable with respect to $(X,A)$. Their products with $D$ also have expectation zero. All products are integrable because the outcomes are bounded and $\delta_0,\delta_1\in L^2(P)$. Thus $\E\{\phi_R(Z;h^*)D(Z)\}=0$.

Finally, $A$ is independent of $X$, and
\[
\E\{(1-A/\pi_1)^2\}=\frac{\pi_0}{\pi_1}.
\]
It follows that
\[
\Var\{\phi_R(Z;h)\}
=V_R^{h^*}+\E\{D(Z)^2\}
=V_R^{h^*}+\frac{\pi_0}{\pi_1p_0^2}\|\epsilon[h]\|_{L^2(P)}^2.
\]
Since the multiplier of the squared norm is positive, equality with $V_R^{h^*}$ is equivalent to $\epsilon[h]=0$ almost surely.
\end{proof}

For completeness, if $\phi_a(Z;h)=\psi^{(a)}(Z;h)-p_a$, direct expansion also gives the usual expression
\[
V_R^h=\frac{\Var\{\phi_1(Z;h)\}}{p_0^2}
+\frac{p_1^2\Var\{\phi_0(Z;h)\}}{p_0^4}
-\frac{2p_1\Cov\{\phi_1(Z;h),\phi_0(Z;h)\}}{p_0^3}.
\]

\subsection{The direct hybrid RR-AIPW estimator}

\begin{proof}[Proof of Theorem~\ref{thm:hybrid-rr}]
For the trial candidate, Proposition~2 gives
\[
\sqrt n(\hat R_{\RRAIPW,1}-R)
=\frac{1}{\sqrt n}\sum_{i=1}^n\phi_{R,1}(Z_i)+o_{P^*}(1).
\]
For $\ell\geq2$, $g_\ell$ is fixed and square-integrable, so the same argument gives
\[
\sqrt n(\hat R_{\RRAIPW,\ell}-R)
=\frac{1}{\sqrt n}\sum_{i=1}^n\phi_{R,\ell}(Z_i)+o_P(1).
\]
Because the number of candidates is fixed,
\[
\sqrt n(\hat{\boldR}-R\boldone)
=\frac{1}{\sqrt n}\sum_{i=1}^n\boldphi_R(Z_i)+o_{P^*}(1)
\rightsquigarrow N(0,\Omega).
\]
For fixed $\lambda\in\Lambda$,
\[
\sqrt n(\lambda^{\T}\hat{\boldR}-R)
\rightsquigarrow N(0,\lambda^{\T}\Omega\lambda).
\]
The constrained quadratic minimization has solution
\[
\lambda^*=\frac{\Omega^{-1}\boldone}{\boldone^{\T}\Omega^{-1}\boldone},
\]
obtained from the Lagrange first-order condition $2\Omega\lambda-\eta\boldone=0$.

The implemented covariance estimator satisfies
\[
\hat\Omega_{\rm reg}-\Omega
=(\hat\Omega-\Omega)+n^{-1/2}\bar\omega_R I_m.
\]
Since $\hat\Omega\to_P\Omega$, continuity of the trace gives $\bar\omega_R=m^{-1}\operatorname{tr}(\hat\Omega)=O_P(1)$, and hence $n^{-1/2}\bar\omega_R=o_P(1)$. Thus $\hat\Omega_{\rm reg}\to_P\Omega$. Because $\Omega$ is nonsingular, the continuous mapping theorem gives $\hat\lambda\to_P\lambda^*$.

Both $\hat\lambda$ and $\lambda^*$ have entries summing to one. Hence
\[
\sqrt n\{\hat R_{\HRRAIPW}- (\lambda^*)^{\T}\hat{\boldR}\}
=(\hat\lambda-\lambda^*)^{\T}\sqrt n(\hat{\boldR}-R\boldone)=o_P(1),
\]
because the first factor is $o_P(1)$ and the second is $O_P(1)$. Slutsky's theorem gives
\[
\sqrt n(\hat R_{\HRRAIPW}-R)
\rightsquigarrow N\{0,(\lambda^*)^{\T}\Omega\lambda^*\}.
\]
For the domination result, let $e_\ell$ be the $\ell$th canonical basis vector. Since $e_\ell\in\Lambda$,
\[
(\lambda^*)^{\T}\Omega\lambda^*
\leq e_\ell^{\T}\Omega e_\ell
=\Omega_{\ell\ell}
\quad (\ell=1,\ldots,k+1).
\]
\end{proof}

\subsection{Efficiency of the direct hybrid estimator}
\begin{corollary}[Efficiency of the direct hybrid estimator]
\label{cor:efficiency}
Under the conditions of Theorem~\ref{thm:hybrid-rr}, let $\epsilon_\ell=\epsilon[g_\ell]$. For every $\lambda\in\Lambda$,
\begin{equation}
\label{eq:hybrid-vardecomp}
\lambda^{\T}\Omega\lambda
=V_R^{h^*}+\frac{\pi_0}{\pi_1p_0^2}
\left\|\sum_{\ell=1}^{k+1}\lambda_\ell\epsilon_\ell\right\|_{L^2(P)}^2.
\end{equation}
The direct hybrid estimator attains the efficiency bound if and only if there exists $\lambda\in\Lambda$ with $\sum_{\ell=1}^{k+1}\lambda_\ell\epsilon_\ell=0$ almost surely.
\end{corollary}

Consequently, the oracle weights also minimize the squared $L^2(P)$ norm of the combined error. The combinations $\sum_\ell\lambda_\ell\epsilon_\ell$, with $\lambda\in\Lambda$, form the affine hull of the error functions in $L^2(P)$; the minimizing combination is the element closest to zero. The hybrid is efficient if $h^\dagger=h^*$ or $f_j=h^*$ for some $j$, but efficiency can also result from cancellation of the errors of several misspecified candidates. 
\begin{proof}[Proof of Corollary~1]
For $\lambda\in\Lambda$, Proposition~3 and $\sum_\ell\lambda_\ell=1$ give
\[
\lambda^{\T}\boldphi_R(Z)
=\phi_R(Z;h^*)+
\frac{1-A/\pi_1}{p_0}\sum_{\ell=1}^{k+1}\lambda_\ell\epsilon_\ell(X).
\]
The orthogonality argument in that proposition applies to each $\epsilon_\ell$ and hence to their linear combination. Taking variances proves equation~\eqref{eq:hybrid-vardecomp}. Theorem~1 identifies $\lambda^*$ as a minimizer over $\Lambda$. Therefore its variance equals $V_R^{h^*}$ precisely when the squared norm at the minimum is zero. If any $\lambda\in\Lambda$ gives a zero combined error, the minimum is zero; conversely, if the minimum is zero, $\lambda^*$ is such a vector. This proves the equivalence.
\end{proof}

\subsection{Equivalence of the two single-model RR-AIPW representations}

\begin{proof}[Proof of Proposition~\ref{prop:direct-plugin-equivalence}]
For any $h$ with $\hat p_0(h)\neq0$,
\[
1+\frac{\hat\theta_{\rm AIPW}(h)}{\hat p_0(h)}
=1+\frac{\hat p_1(h)-\hat p_0(h)}{\hat p_0(h)}
=\frac{\hat p_1(h)}{\hat p_0(h)}
=\hat R_{\RRAIPW}(h).
\]
\end{proof}

\subsection{The componentwise hybrid RR-AIPW estimator}

\begin{proof}[Proof of Theorem~\ref{thm:separate-plug-in}]
Finite-dimensional stacking gives
\begin{align*}
\sqrt n(\hat{\boldtheta}-\theta\boldone)
&=\frac{1}{\sqrt n}\sum_{i=1}^n\boldphi_\theta(Z_i)+o_{P^*}(1),\\
\sqrt n(\hat{\boldp}_0-p_0\boldone)
&=\frac{1}{\sqrt n}\sum_{i=1}^n\boldphi_0(Z_i)+o_{P^*}(1).
\end{align*}
The same vanishing-ridge argument used for $\hat\Omega_{\rm reg}$ shows that the regularized empirical covariance matrices for $\boldphi_\theta$ and $\boldphi_0$ remain consistent for $\Sigma_\theta$ and $\Sigma_0$. Therefore $\hat\lambda_\theta\to_P\lambda_\theta^*$ and $\hat\lambda_0\to_P\lambda_0^*$.

Since $\hat\lambda_\theta,\lambda_\theta^*,\hat\lambda_0,\lambda_0^*\in\Lambda$,
\begin{align*}
\hat\lambda_\theta^{\T}\hat{\boldtheta}-(\lambda_\theta^*)^{\T}\hat{\boldtheta}
&=(\hat\lambda_\theta-\lambda_\theta^*)^{\T}(\hat{\boldtheta}-\theta\boldone),\\
\hat\lambda_0^{\T}\hat{\boldp}_0-(\lambda_0^*)^{\T}\hat{\boldp}_0
&=(\hat\lambda_0-\lambda_0^*)^{\T}(\hat{\boldp}_0-p_0\boldone).
\end{align*}
Thus
\begin{align*}
\sqrt n\{\hat\lambda_\theta^{\T}\hat{\boldtheta}-\theta\}
&=\frac{1}{\sqrt n}\sum_{i=1}^n(\lambda_\theta^*)^{\T}\boldphi_\theta(Z_i)+o_{P^*}(1),\\
\sqrt n\{\hat\lambda_0^{\T}\hat{\boldp}_0-p_0\}
&=\frac{1}{\sqrt n}\sum_{i=1}^n(\lambda_0^*)^{\T}\boldphi_0(Z_i)+o_{P^*}(1).
\end{align*}
Apply the delta method to $g(u,v)=1+u/v$ at $(\theta,p_0)$. Its gradient at this point is $(p_0^{-1},-\theta p_0^{-2})$, so
\[
\sqrt n(\hat R_{\SepRRAIPW}-R)
=\frac{1}{\sqrt n}\sum_{i=1}^n
\left[
\frac{(\lambda_\theta^*)^{\T}\boldphi_\theta(Z_i)}{p_0}
-
\frac{\theta(\lambda_0^*)^{\T}\boldphi_0(Z_i)}{p_0^2}
\right]
+o_{P^*}(1).
\]
The central limit theorem gives the stated normal limit. Expanding the variance gives
\begin{align*}
\widetilde V_{\rm sep}
={}&
\frac{(\lambda_\theta^*)^{\T}\Sigma_\theta\lambda_\theta^*}{p_0^2}
+
\frac{\theta^2(\lambda_0^*)^{\T}\Sigma_0\lambda_0^*}{p_0^4}\\
&-
\frac{2\theta(\lambda_\theta^*)^{\T}\Sigma_{\theta0}\lambda_0^*}{p_0^3}.
\end{align*}
\end{proof}

%% file: main.bib
@article{Bartolomeis2025,
  author  = {De Bartolomeis, Piersilvio and Abad, Javier and Wang, Guanbo and Donhauser, Konstantin and Duch, Raymond M. and Yang, Fanny and Dahabreh, Issa J.},
  title   = {Efficient randomized experiments using foundation models},
  journal = {arXiv:2502.04262},
  year    = {2025},
  note    = {Accepted for presentation at the Thirty-Ninth Conference on Neural Information Processing Systems}
}

@article{Lin2024,
  author  = {Lin, Xi and Tarp, Jens Magelund and Evans, Robin J.},
  title   = {Data fusion for efficiency gain in {ATE} estimation: a practical review with simulations},
  journal = {arXiv:2407.01186},
  year    = {2024}
}

@article{Sistrom2004,
  author  = {Sistrom, Christopher L. and Garvan, Cynthia W.},
  title   = {Proportions, odds, and risk},
  journal = {Radiology},
  year    = {2004},
  volume  = {230},
  number  = {1},
  pages   = {12--19},
  doi     = {10.1148/radiol.2301031028}
}

@book{Carneiro2011,
  author    = {Carneiro, Ilona and Howard, Natasha},
  title     = {Introduction to Epidemiology},
  edition   = {2},
  publisher = {Open University Press},
  address   = {Maidenhead, UK},
  year      = {2011}
}

@article{Poulet2025,
  author  = {Poulet, Pierre-Emmanuel and Tran, Maylis and Tezenas du Montcel, Sophie and Dubois, Bruno and Durrleman, Stanley and Jedynak, Bruno},
  title   = {Prediction-powered inference for clinical trials: application to linear covariate adjustment},
  journal = {BMC Med Res Method},
  year    = {2025},
  volume  = {25},
  pages   = {204},
  doi     = {10.1186/s12874-025-02647-6}
}

@article{Luedtke2019,
  author  = {Luedtke, Alex and Carone, Marco and van der Laan, Mark J.},
  title   = {An omnibus non-parametric test of equality in distribution for unknown functions},
  journal = {J R Stat Soc Ser B},
  year    = {2019},
  volume  = {81},
  number  = {1},
  pages   = {75--99},
  doi     = {10.1111/rssb.12299}
}

@inproceedings{Hussain2023,
  author    = {Hussain, Zeshan and Shih, Ming-Chieh and Oberst, Michael and Demirel, Ilker and Sontag, David},
  title     = {Falsification of internal and external validity in observational studies via conditional moment restrictions},
  booktitle = {Proceedings of the 26th International Conference on Artificial Intelligence and Statistics},
  series    = {Proc Mach Learn Res},
  volume    = {206},
  pages     = {5869--5898},
  publisher = {PMLR},
  year      = {2023}
}

@article{Dahabreh2024,
  author  = {Dahabreh, Issa J. and Matthews, Anthony and Steingrimsson, Jon and Scharfstein, Daniel O. and Stuart, Elizabeth A.},
  title   = {Using trial and observational data to assess effectiveness: trial emulation, transportability, benchmarking, and joint analysis},
  journal = {Epidemiol Rev},
  year    = {2024},
  volume  = {46},
  number  = {1},
  pages   = {1--16},
  doi     = {10.1093/epirev/mxac011}
}

@inproceedings{Bartolomeis2024a,
  author    = {De Bartolomeis, Piersilvio and Abad Martinez, Javier and Donhauser, Konstantin and Yang, Fanny},
  title     = {Hidden yet quantifiable: a lower bound for confounding strength using randomized trials},
  booktitle = {Proceedings of the 27th International Conference on Artificial Intelligence and Statistics},
  series    = {Proc Mach Learn Res},
  volume    = {238},
  pages     = {1045--1053},
  publisher = {PMLR},
  year      = {2024}
}

@inproceedings{Bartolomeis2024b,
  author    = {De Bartolomeis, Piersilvio and Abad, Javier and Donhauser, Konstantin and Yang, Fanny},
  title     = {Detecting critical treatment effect bias in small subgroups},
  booktitle = {Proceedings of the Fortieth Conference on Uncertainty in Artificial Intelligence},
  series    = {Proc Mach Learn Res},
  volume    = {244},
  pages     = {943--965},
  publisher = {PMLR},
  year      = {2024}
}

@inproceedings{Kallus2018,
  author    = {Kallus, Nathan and Puli, Aahlad Manas and Shalit, Uri},
  title     = {Removing hidden confounding by experimental grounding},
  booktitle = {Advances in Neural Information Processing Systems},
  volume    = {31},
  pages     = {10911--10920},
  year      = {2018}
}

@article{Rosenman2023,
  author  = {Rosenman, Evan T. R. and Basse, Guillaume and Owen, Art B. and Baiocchi, Mike},
  title   = {Combining observational and experimental datasets using shrinkage estimators},
  journal = {Biometrics},
  year    = {2023},
  volume  = {79},
  number  = {4},
  pages   = {2961--2973},
  doi     = {10.1111/biom.13827}
}

@article{Laan2024,
  author  = {van der Laan, Mark and Qiu, Sky and Tarp, Jens Magelund and van der Laan, Lars},
  title   = {Adaptive-TMLE for the average treatment effect based on randomized controlled trial augmented with real-world data},
  journal = {J Causal Inference},
  year    = {2026},
  volume  = {14},
  number  = {1},
  pages   = {1--36},
  doi     = {10.1515/jci-2024-0025}
}

@article{Dang2025,
  author  = {Dang, Lauren Eyler and Tarp, Jens Magelund and Abrahamsen, Trine Julie and Kvist, Kajsa and Buse, John B. and Petersen, Maya and van der Laan, Mark},
  title   = {Experiment-selector cross-validated targeted maximum likelihood estimator for hybrid RCT-external data studies},
  journal = {J Causal Inference},
  year    = {2025},
  volume  = {13},
  number  = {1},
  pages   = {1--33},
  doi     = {10.1515/jci-2024-0041}
}

@book{tsiatis2006semiparametric,
  title     = {Semiparametric Theory and Missing Data},
  author    = {Tsiatis, Anastasios A.},
  year      = {2006},
  publisher = {Springer},
  address   = {New York}
}
